\documentclass[11pt]{article}
\usepackage{amsthm}
\usepackage{amssymb}
\usepackage{amsmath}
\usepackage{hyperref}
\usepackage[all,cmtip]{xy}
\newtheorem{theorem}{Theorem}

\newtheorem{lemma}[theorem]{Lemma}
\theoremstyle{definition}

\DeclareMathOperator{\diam}{diam}

\newcommand{\fin}{\mathrm{f}}

\DeclareMathOperator{\vol}{Vol}
\begin{document}
\title{A New Proof of an Arithmetic Riemann-Roch Theorem}
\author{Sam Mundy}
\date{}
\maketitle
\begin{abstract}
In this paper, we give a new proof of an arithmetic analogue of the Riemann-Roch Theorem, due originally to Serge Lang. Lang's result was first proved using the lattice point geometry of Minkowski. By contrast, our proof is completely adelic. It has the conceptual advantage that it uses a different analogue of the Riemann-Roch theorem proved by Tate in his thesis, in a manner similar to the proof of the Riemann-Roch theorem for curves over finite fields which uses Tate's theorem. Thus our proof provides a bridge between a lot of the Riemann-Roch theory that exists in arithmetic.
\end{abstract}
\section*{Introduction}
Throughout this paper, $K$ is a number field of degree $n$, with $r$ real embeddings and $2s$ complex embeddings, so that $n=r+2s$. The discriminant of $K$ is denoted $\Delta_K$. The ring of integers in $K$ is written $\mathcal{O}_K$, and the group of fractional ideals of it is denoted $J(\mathcal{O}_K)$. The norm of a fractional ideal $\mathfrak{a}$ is denoted $\mathbb{N}\mathfrak{a}$ and is defined, when $\mathfrak{a}$ is an ideal, to be the cardinality of the set $\mathcal{O}_K/\mathfrak{a}$. The set of places of $K$ is denoted $V_K$, and it is the disjoint union of the finite and infinite places, denoted respectively $V_\fin$ and $V_\infty$. The completion of $K$ at a place $v\in V_K$ is denoted $K_v$. If the place $v$ is finite, then the valuation ring of $K_v$ is denoted $\mathcal{O}_v$.\\
\indent The ring of adeles of $K$ is denoted $\mathbb{A}_K$, and the group of ideles, $\mathbb{I}_K$. The norm of an idele is $x$ is denoted $\Vert x\Vert$, and the set of ideles of norm $1$ is written $\mathbb{I}_K^1$. The embedding $K\subset\mathbb{A}_K$ is diagonal and restricts to an embedding $K^\times\subset\mathbb{I}_K^1$ by the product formula. Then we have the standard results that $K$ and $K^\times$ embed discretely and cocompactly in both cases, under the usual topologies on $\mathbb{A}_K$ and $\mathbb{I}_K$. The fact that $K$ and $K^\times$ are cocompact under their respective embeddings means that $\mathbb{A}_K/K$ and $\mathbb{I}_K^1/K^\times$ are compact, each with the quotient topology.\\
\indent The groups $\mathbb{A}_K$, $\mathbb{I}_K$ and $\mathbb{I}_K^1$ are all locally compact as abelian topological groups. We will apply the machinery of Fourier analysis on locally compact abelian groups to these, as Tate did in his famous thesis. His thesis was published in Cassels and Fr\"{o}hlich \cite{CF}, and is exposed in more generality (i.e., for all local and global fields) in Ramakrishnan and Valenza \cite{RV}. For a purely analytic account of the Fourier analysis of locally compact abelian groups, see Folland \cite{Fol}. This Fourier analysis constitutes the basic tools of this paper, and the only theorems from Tate's thesis we will use are his Riemann-Roch Theorem, and the theorems necessary to state it. We will recall these in a moment.\\
\indent First, we must fix a character $\psi$ on $\mathbb{A}_K$. That is, $\psi$ is a continuous homomorphism from $\mathbb{A}_K$ to the circle $S^1\subset\mathbb{C}^\times$ with the usual euclidean topology. We require only that this character is nontrivial and has the property that it is trivial on $K$. A choice of such a character is immaterial for our purposes, and we need only its existence. A character like this is easy to construct (take, for instance, what is called the ``standard character" in Ramakrishnan and Valenza \cite{RV}), but we will not construct one here.\\
\indent The map $x\mapsto(y\mapsto\psi(xy))$ is a topological isomorphism between $\mathbb{A}_K$ and its dual group. Under this identification, $K=K^\perp$, where $K^\perp$ is defined to be the set of characters on $\mathbb{A}_K$ which are trivial on $K$.\\
\indent The Haar measure will be taken to be the one which is self dual. It comes from measures on the local components of $\mathbb{A}_K$ which satisfy similar dualities. Namely it comes from: the Lebesgue measure on $\mathbb{R}$; twice the Lebesgue measure on $\mathbb{C}$; and, when $v\in V_\fin$, it comes from $(\mathbb{N}\mathfrak{D})^{-1/2}$ times the measure on $K_v$ which gives $\mathcal{O}_v$ measure $1$. Here the $\mathfrak{D}$ is the local different. This choice has the consequence that the product of the measures of all valuation rings at the finite places of $K$ equals $\vert\Delta_K\vert^{-1/2}$.\\
\indent All Fourier transforms will be taken with respect to $\psi$ and the above Haar measure, which we will denote by $dx$ if it is being integrated, and by $\vol$ if it is not. We denote the Fourier transform with a hat:
\[\hat{f}(y)=\int_{\mathbb{A}_K}f(x)\psi(xy)\,dx.\]
Tate's Riemann-Roch Theorem is
\begin{theorem}[Riemann-Roch Theorem of Tate]
Let $K$ be a global field. Let $f\in L^1(\mathbb{A}_K)$ be continuous. Assume that the series $\sum_{\alpha\in K}f(y(x+\alpha))$ converges for all $x\in\mathbb{A}_K$ and all $y\in\mathbb{I}_K$, uniformly in $x$ on compact subsets of $\mathbb{A}_K$ and uniformly in $y$ on compact subsets of $\mathbb{I}_K$. Assume also that $\sum_{\alpha\in K}\hat{f}(\alpha y)$ converges for all $y\in\mathbb{I}_K$. Then
\[\sum_{\alpha\in K}f(\alpha y)=\frac{1}{\Vert y\Vert}\sum_{\alpha\in K}\hat{f}\left(\frac{\alpha}{y}\right).\]
\end{theorem}
The adeles decompose as a product $\mathbb{A}_K=\mathbb{A}_\fin\times(\mathbb{R}^r\times\mathbb{C}^s)$, where the first component is the restricted direct product of all completions of $K$ at finite places, with respect to their valuation rings. We will call a function on $\mathbb{A}_K$ an adelic Schwartz function if it is a product of a continuous function with compact support on $\mathbb{A}_\fin$, and a Schwartz function on $\mathbb{R}^r\times\mathbb{C}^s$, in the usual sense. Any adelic Schwartz function satisfies the hypotheses of Tate's theorem (see \cite{RV}), and all functions to which we apply this theorem will obviously be adelic Schwartz.\\
\indent We recall the notion of replete ideal (See Neukirch \cite{Neu}, chapter 3). A replete ideal is an element of the group
\[J(\overline{\mathcal{O}}_K)=J(\mathcal{O}_K)\times\mathbb{R}_{>0}^{r+s},\]
where $\mathbb{R}_{>0}$ is the group of positive real numbers under multiplication. We think of elements of this group as ideals with an infinite part, and we think of the components $\mathbb{R}_{>0}$ as being indexed by $V_\infty$. With this in mind, let $\mathfrak{a}$ be a replete ideal with finite and infinite parts $\mathfrak{a}_\fin$ and $(n_v)_{v\in V_\infty}$, respectively. Embed $\mathfrak{a}_\fin$ into $\mathbb{R}^r\times\mathbb{C}^s$ as is usually done in the Minkowski theory. We define $H^0(\mathfrak{a})$ to be the set of elements of $\alpha\in\mathfrak{a}_\fin$, viewed in $\mathbb{R}^r\times\mathbb{C}^s$, such that $\vert\alpha\vert_v\leq n_v^{-1}$ for all $v\in V_\infty$, where $\vert\cdot\vert_v$ is, of course, the absolute value corresponding to the place $v$.\\
\indent Define the norm of the replete ideal $\mathfrak{a}=\mathfrak{a}_\fin\times(n_v)_{v\in V_\infty}$ to be
\[\mathbb{N}\mathfrak{a}=\mathbb{N}\mathfrak{a}_\fin\times\prod_{v\in V_\infty}n_v^{f_v},\]
where $f_v$ is $1$ or $2$ depending on whether $v$ is real or complex, respectively.\\
\indent The next theorem is the result of Lang. It was first proven by him in the sixties, heavily using tools from Minkowski theory, and it was published in his book \cite{LangANT}.
\begin{theorem}[Riemann-Roch Theorem For Number Fields]
As $\mathfrak{a}$ ranges through $J(\bar{\mathcal{O}}_K)$, we have the estimate
\[\vert H^0(\mathfrak{a}^{-1})\vert=\frac{2^r(2\pi)^s}{\sqrt{\vert\Delta_K\vert}}\mathbb{N}\mathfrak{a}+O((\mathbb{N}\mathfrak{a})^{1-\frac{1}{n}}).\]
Here, as usual, the $O$-term denotes a function which is bounded by a constant times its argument.
\end{theorem}
We prove this theorem adelically using the Fourier analysis described above. To do this, we will let $B\subset\mathbb{A}_K$ be the product of the closed unit balls in $K_v$, $v\in V_K$, and let $\chi_B$ be its characteristic function. The main result will be to prove that as $a$ ranges through $\mathbb{I}_K$, we have the estimate
\[\sum_{\alpha\in K}\chi_B(\alpha a^{-1})=\frac{2^r(2\pi)^s}{\sqrt{\vert\Delta_K\vert}}\Vert a\Vert+O(\Vert a\Vert^{1-\frac{1}{n}}).\]
\indent The first instinct could be to apply Tate's Riemann-Roch theorem to $\chi_B$, like one does in the case of function fields to prove the Riemann-Roch theorem for curves over finite fields (see Ramakrishnan and Valenza \cite{RV}). This would not work, however, since the infinite components of $\widehat{\chi}_B$ violate the hypotheses of that theorem for reasons related to continuity. Instead, we will convolve $\chi_B$ with a bump function, use Riemann-Roch to estimate a sum involving the resulting function, and compare this sum to the one above. The estimations we need will rely on a notion of surface area for adelic regions. All of this we be explained now.
\section*{The Proof}
\indent Fix for now a \it region \rm $D\subset\mathbb{A}_K$, by which we mean $D$ is compact and equal to the closure of its interior, and the finite component of $D$ is compact and open. Given another region $E\subset\mathbb{A}_K$, we define
\[E_D=\lbrace x+(y_1-y_0)\mid x\in E,\, y_1,y_0\in D\rbrace.\]
This region is simply the union of all translates of $D$ which intersect $E$.\\
\indent Actually, we will want to consider this construction when $D$ is skewed. Given an idele $a\in\mathbb{I}_K$ and a subset $X\subset\mathbb{A}_K$, we define $aX=\lbrace ax\mid x\in X\rbrace$. If $t>0$ is a real number, we may also view $t$ as the idele whose finite components are all $1$ and whose infinite components are all $t$, so that the expression $tD$ makes sense and is a region. Then we define the \it adelic surface area \rm of $E$ with respect to $D$ to be the derivative
\[S_{D}(E)=\lim_{t\to 0^+}\frac{\vol(E_{tD})-\vol(E)}{t},\]
if it exists.
\begin{lemma}
\label{pf:C}
Let $D,E$ be regions for which $S_{aD}(E)$ exists for all norm $1$ ideles $a$. Let $C:[0,1]\times\mathbb{I}_K^1\to\mathbb{R}$ be defined by
\[C(t,a)=\begin{cases}
\frac{\vol(E_{t,aD})-\vol(E)}{t}&\textrm{if}\,\,t\ne 0\\
S_{aD}(E)&\textrm{if}\,\, t=0.
\end{cases}\]
Then $C$ is continuous and defined everywhere.
\end{lemma}
\begin{proof}
Let $a=(a_v)_{v\in V_K}$ be a norm $1$ idele. Since volume is translation invariant, we may assume $0\in D$, and hence in $aD$. Then there is a finite subset $S\subset V_K$, including the infinite places, such that $aD$ is the product $\prod_{v\notin S}\mathcal{O}_v$ away from the places in $S$. Let $\epsilon>0$ and let $b=(b_v)$ be a norm $1$ idele such that: $\Vert b_v/a_v-1\Vert<\epsilon$ for all infinite places $v$; $\Vert b_v/a_v-1\Vert$ is so small for $v\in S\cap V_\fin$ that skewing $aD$ at the place $v$ does not change the region $aD$; and $b_v\in\mathcal{O}_v^\times$ for $v\notin S$. The set of all $b$ satisfying these conditions is open in $\mathbb{I}_K^1$.\\
\indent Now these conditions on $b$ imply that $bD=(b/a)aD$ is such that $(1+\epsilon)^{-1}bD\subset aD$ and $(1+\epsilon)^{-1}aD\subset bD$. Thus
\[(1+\epsilon)^{-1}aD\subset bD\subset(1+\epsilon)aD,\]
and hence, for $t>0$,
\[E_{(1+\epsilon)^{-1}taD}\subset E_{tbD}\subset E_{(1+\epsilon)taD}.\]
Therefore, taking volumes, subtracting the volume of $E$ and dividing by $t$, we get
\[(1+\epsilon)^{-1}C((1+\epsilon)^{-1}t,a)\leq C(t,b)\leq (1+\epsilon)C((1+\epsilon)t,a)\]
This proves that $C$ is continuous on $(0,1]\times\mathbb{I}_K^1$ and taking the limit of the above estimate as $t\to 0$ gives continuity everywhere.
\end{proof}
In what follows, $B\subset\mathbb{A}_K$ is the product of the closed unit balls at each place, $B=\prod_{v\in V_\fin}\mathcal{O}_v\times\prod_{v\in V_\infty}\overline{B(0,1)}$. This is a region in $\mathbb{A}_K$. Also, $D$ will denote the closure of a fixed fundamental domain of $\mathbb{A}_K$ modulo $K$ for which $S_{aD}(E)$ exists for all norm $1$ ideles $a$. Also, we let $\varphi$ be a continuous function with support in $D$ and total integral $1$, such that the series $\sum_{\alpha\in K}\vert\widehat{\varphi}(\alpha b)\vert$ converges and is continuous for ideles $b$. We will exhibit such a pair $D,\varphi$ in a moment, but for now, we will assume that they exist.\\
\indent Now for two functions $g,h$ on $\mathbb{A}_K$, define their convolution $g*h$ as usual:
\[g*h(x)=\int_{\mathbb{A}_K}g(y)h(x-y)\,dy.\]
Then $(g*h)\,\,\widehat{}\,\,=\hat{g}\hat{h}$. Let $f=\chi_B*\varphi$.
\begin{lemma}
\label{pf:d}
Let $a\in\mathbb{I}_K^1$ and $t>1$. There is a constant $c_1$, depending only on $D$ and $\varphi$, such that
\[\sum_{\alpha\in K}\chi_B(\alpha(ta)^{-1})\leq\sum_{\alpha\in K}f(\alpha(ta)^{-1})+c_1.\]
\end{lemma}
\begin{proof}
We have
\[f((ta)^{-1}\alpha)=\int\chi_B(x)\varphi((ta)^{-1}\alpha-x)\,dx=\int_{taB}\varphi((ta)^{-1}(\alpha-x))t^{-n}\,dx.\]
Since $\varphi$ has total integral $1$, this integral is equal to $\chi_B((ta)^{-1}\alpha)$ whenever the support of $\varphi((ta)^{-1}(\alpha-x))$ is contained completely inside or outside $taB$, i.e., it does not intersect the boundary of $taB$. Now the support of $\varphi((ta)^{-1}(\alpha-x))$ intersects the boundary of $taB$ only when $ta(\alpha-D)$ intersects the boundary of $taB$, which happens if and only if $\alpha-D$ intersects the boundary of $B$, and this happens only finitely many times. Thus $f((ta)^{-1}\alpha)\ne\chi_B((ta)^{-1}\alpha)$ for finitely many $\alpha$, say $c$ of them, and the difference is at most $1$. Thus we are done if we take $c_1=c$.
\end{proof}
\begin{lemma}
\label{pf:e}
There is a continuous function $C_1$ on $\mathbb{I}_K^1$ such that, in the notation of Lemma \ref{pf:C},
\[\sum_{\alpha\in K^\times}f(at\alpha)\leq t^{-1}C_1(a)C(t^{-1},a^{-1}).\]
\end{lemma}
\begin{proof}
We have
\begin{align*}
\hat{f}(at\alpha)&=\widehat{\chi}_B(at\alpha)\widehat{\varphi}(at\alpha)\\
&=\widehat{\varphi}(at\alpha)\int_B\psi(at\alpha x)\,dx\\
&=t^{-n}\widehat{\varphi}(at\alpha)\int_{taB}\psi(\alpha x)\,dx.
\end{align*}
Now since the integral of $\psi(\alpha x)$ is zero over any translate $D$ for $\alpha\in K^\times$ (because the integral of a character over a compact group is zero), the integral $\int_{taB}\psi(\alpha x)\,dx$ is equal to $-\int_{E\backslash taB}\psi(\alpha x)\,dx$, where $E$ is the union of all $K$-translates of $D$ which intersect $taB$. Since the maximum value of $\psi$ is $1$, this is smaller in absolute value than $\vol(E\backslash taB)$. But by the definition of $E$, we know that $E\subset(taB)_D$. So 
\[\vol(E\backslash taB)\leq\vol((taB)_D)-\vol(taB)=t^n(\vol(B_{t^{-1}a^{-1}D})-\vol(B))=t^{n-1}C(t^{-1},a^{-1}).\]
Thus
\[\vert\hat{f}(at\alpha)\vert\leq t^{-1}C(t^{-1},a^{-1})\vert\varphi(at\alpha)\vert.\]
Summing over all $\alpha\in K^\times$ gives the result since the series $\sum_{\alpha\in K^\times}\vert\hat{\varphi}(at\alpha)\vert$ is continuous by assumption and must eventually decrease as $t$ increases.
\end{proof}
Now we exhibit a partcular $D$ and $\varphi$ for which the above theory works.
\begin{lemma}
Let $D$ have finite component $\prod_{v\in V_\fin}\mathcal{O}_v$ and infinite component equal to a fundamental parallelepiped of $\mathcal{O}_K$ in $\mathbb{R}^n$. Let $\varphi$ have finite component the characteristic function of $\prod_{v\in V_\fin}\mathcal{O}_v$ and infinite component any smooth function supported in the fundamental parallelepiped of $\mathcal{O}_K$ in $\mathbb{R}^n$, such that the total integral of $\varphi$ is $1$. Then:\\
\indent (1) $D$ is a fundamental domain for $\mathbb{A}_K$ modulo $K$ and is a region;\\
\indent (2) $S_{aD}(B)$ exists and is finite for all norm $1$ ideles $a$;\\
\indent (3) The series $\sum_{\alpha\in K}\vert\widehat{\varphi}(\alpha b)\vert$ converges and is continuous for ideles $b$.
\end{lemma}
\begin{proof}
The assertion (1) is trivial and well known. Assertion (3) follows from the fact that $\widehat{\varphi}(\alpha b)$ is only nonzero for $\alpha$ in a certain fractional ideal, and from the fact that the infinite component of $\widehat{\varphi}$ is a Schwartz function, because the infinite component of $\varphi$ is.\\
\indent For (2), let $P$ be the infinite component of $aD$, and let $B^\prime$ be the infinite component of $B$, so $B^\prime$ is a product of $r$ intervals and $s$ discs. The finite components of $B_{t,aD}$ are the same for all $t$, so we only need to show that the derivative of $\vol(B_{t,P}^\prime)$ exists, where the regions we are dealing with are now in $\mathbb{R}^n$, and $B_{t,P}^\prime$ has the same meaning as above, but is a region in $\mathbb{R}^r\times\mathbb{C}^s$.\\
\indent For this, let $O$ be in the interior of $B^\prime$ and choose spherical coordinates $(t,\theta)=(t,\theta_1,\dotsc,\theta_{n-1})$ around $O$. Let $f(t,\theta)$ be the function which sends $(t,\theta)$ to the distance from the point $O$ to the boundary of $B_{t,P}^\prime$. This is single-valued because each $B_{t,P}^\prime$ is convex and it is continuous because its graph, which is the boundary of $B_{t,P}^\prime$, is connected. By definition, the volume $\vol(B_{t,P}^\prime)$ is equal to the integral
\[\int_{S^{n-1}}f(t,\theta)\,d\theta.\]
for an appropriate normalization of the form $d\theta$, where $S^{n-1}$ is the $(n-1)$-sphere about $O$. Thus we want to show that the limit
\[\lim_{t\to 0^+}\frac{1}{t}\int_{S^{n-1}}(f(t,\theta)-f(0,\theta))\,d\theta\]
exists.\\
\indent To do this, we show that the difference quotient
\[\frac{f(t,\theta)-f(0,\theta)}{t}\]
is increasing and uniformly bounded above in $\theta$. This will prove that the integral
\[\lim_{t\to 0^+}\int\frac{f(t,\theta)-f(0,\theta)}{t}\,d\theta\]
is increasing and bounded above, and hence has a limit, as desired. So to prove this, we first note that
\[\frac{f(t,\theta)-f(0,\theta)}{t}\leq\frac{\diam(tP)}{t}=\diam P,\]
because the distance between any point on the boundary of $B_{0,P}^\prime=B^\prime$ from the boundary of $B_{t,P}^\prime$ is at most the distance between any two points in $P$, by definition.\\
\indent Now to show that the above difference quotient is increasing, we only need to show that function $f(t,\theta)$ for fixed $\theta$ is concave-down, in the sense of freshman calculus. This means that for any $t$ and any $0<c<1$, we must have
\[cf(t,\theta)\leq f(ct,\theta).\]
We claim that if $v$ is a vector in the direction of $\theta$, if $x\in B^\prime$ is the point on the ray from $O$ in the direction of $\theta$, and $x+v\in B_{t,P}^\prime$, then $x+cv\in B_{ct,P}^\prime$. This assertion implies immediately that $f(t,\theta)$ is concave down in $t$, and its verification will complete the proof of the lemma.\\
\indent Now for such $x$ and $v$, let $y\in B_\infty$ such that $x+v=y+(p_1-p_0)$ for some $p_1,p_0\in tP$. Then we compute
\[x+cv=x+cx-cx+cv=(1-c)x+c(x+v)=(1-c)x+c(y+p_1-p_0)=((1-c)x+cy)+(cp_1-cp_0).\]
But $(1-c)x+cy\in B^\prime$ because $B^\prime$ is convex, and $cp_1,cp_0\in ctP$, so $x+cv\in B_{ct,P}^\prime$. We are done.
\end{proof}
Now we can complete the proof of our main theorem.
\begin{theorem}
As $a$ ranges through $\mathbb{I}_K$, we have the estimate
\[\sum_{\alpha\in K}\chi_B(\alpha a^{-1})=\frac{2^r(2\pi)^s}{\sqrt{\vert\Delta_K\vert}}\Vert a\Vert+O(\Vert a\Vert^{1-\frac{1}{n}}).\]
\end{theorem}
\begin{proof}
We apply Tate's Riemann-Roch theorem to $f$: In the notation of Lemmas \ref{pf:C}, \ref{pf:d}, \ref{pf:e}, we have
\begin{align*}
\sum_{\alpha\in K}\chi_B(\alpha(ta)^{-1})&\leq\sum_{\alpha\in K}f(\alpha(ta)^{-1})+c_1\\
&=\Vert ta\Vert\sum_{\alpha\in K}f(\alpha ta)+c_1\\
&\leq\Vert ta\Vert\hat{f}(0)+c_1+t^{-1}\Vert ta\Vert C_1(a)C(t^{-1},a^{-1}).
\end{align*}
Now
\[\hat{f}(0)=\widehat{\varphi}(0)\widehat{\chi}_B(0)=\vol(B)=\frac{2^r(2\pi)^s}{\sqrt{\vert\Delta_K\vert}}\]
and $t^{-1}\Vert ta\Vert=\Vert ta\Vert^{1-\frac{1}{n}}$ because $a$ has norm $1$. Thus we have the desired estimate in $t$ for each $a$ separately. To get a uniform estimate, we note that the quantity
\[\sum_{\alpha\in K}\chi_B(\alpha(ta)^{-1})\]
does not depend on the class of $a$ modulo $K^\times$. Thus we may replace the function $(t,a)\mapsto C_1(a)C(t^{-1},a^{-1})$ by the function $C^\prime:[0,1]\times(\mathbb{I}_K^1/K^\times)\to\mathbb{R}$ defined by 
\[C^\prime(t^{-1},x)=\inf\lbrace C_1(a)C(t^{-1},a^{-1})\mid\pi(a)=x\rbrace\]
where $\pi:\mathbb{I}_K^1\to\mathbb{I}_K^1/K^\times$ is the quotient map. Since $[0,1]\times(\mathbb{I}_K^1/K^\times)$ is compact, the following lemma suffices to complete the proof of the theorem.
\end{proof}
\begin{lemma}
(a) Let $f$ be a positive continuous real valued function on a topological space $X$, and let $\pi:X\to Y$ be a quotient map. Then $F:Y\to\mathbb{R}$ given by $F(y)=\inf\lbrace f(x)\mid x\in\pi^{-1}(y)\rbrace$ is upper semicontinuous on $Y$, which means that the sets $F^{-1}([a,\infty))$ are closed in $Y$ for all $a\in\mathbb{R}$.\\
\indent (b) An upper semicontinuous function on a compact topological space is bounded above by a constant.
\end{lemma}
\begin{proof}
Though this is an easy exercise, we include its proof for lack of a suitable reference.\\
\indent (a) With the notation as in the lemma, we need to show that the sets $F^{-1}((-\infty,a))$ are open. Let $b\in\mathbb{R}$. The set $U=f^{-1}((-\infty,b))\subset X$ is open by continuity of $f$. Then $\pi(U)$ is open in $Y$. But $\pi(U)$ contains all points $y\in Y$ for which there is an $x\in X$ with $\pi(x)=y$ and $f(x)<b$. This means exactly that $F(x)<b$, and so $F^{-1}((-\infty,b))$ is the open set $\pi(U)$.\\
\indent (b) Let $F$ be an upper semicontinuous function on the compact space $Y$. The sets $F^{-1}((-\infty,a))$ for $a\in\mathbb{R}$ form an open cover of $Y$. Thus it has a finite subcover, say $\lbrace F^{-1}((-\infty,a_i))\rbrace$. Let $j$ be such that $a_j$ is maximal amongst the $a_i$'s. Then $F^{-1}((-\infty,a_j))=Y$ since all of the sets $F^{-1}((-\infty,a_i))$ are subsets of this one. Hence $F(y)<a_j$ for all $y\in Y$.
\end{proof}
Now we apply the theorem above to give a proof of the Riemann-Roch theorem for number fields. For $x=(x_v)_{v\in V_K}\in\mathbb{I}_K$, let $\mathfrak{a}_x$ be the replete ideal defined by
\[\mathfrak{a}_x=\prod_{\mathfrak{p}_v\in V_\fin}\mathfrak{p}^{-v_{\mathfrak{p}_v}(x_\mathfrak{p})}\times(\vert x_v\vert)_{v\in V_\infty},\]
where $\mathfrak{p}_v$ denotes the prime ideal associated to the place $v$. Then, by definition, $\mathbb{N}\mathfrak{a}_x=\Vert x\Vert$ and the sum
\[\sum_{\alpha\in K}\chi_B(\alpha x^{-1})\]
counts the number of elements in $H^0(\mathfrak{a}_x^{-1})$. Then we apply our theorem for ideles $x$ to get the Riemann-Roch theorem for replete ideals $\mathfrak{a}_x$. Since the map $x\mapsto\mathfrak{a}_x$ from $\mathbb{I}_K$ to $J(\overline{\mathcal{O}}_K)$ is clearly surjective, we obtain the Riemann-Roch theorem for number fields.

\end{document}